\documentclass[11pt,letterpaper,reqno]{amsart}

\usepackage{geometry}
\usepackage{amsmath}
\usepackage{amsthm}
\usepackage{amssymb}
\usepackage{bbm}

\usepackage{bookmark}
\usepackage{hyperref}
\hypersetup{pdfstartview={FitH}}

\theoremstyle{plain}
\newtheorem{prop}{Proposition}
\newtheorem{thm}[prop]{Theorem}
\newtheorem{lemma}[prop]{Lemma}

\title[Continuous-time double averages]{Pointwise convergence of certain continuous-time double ergodic averages}

\author[M. Christ]{Michael Christ}
\address{Michael Christ\\
        Department of Mathematics\\
        University of California \\
        Berkeley, CA 94720, USA}
\email{mchrist@berkeley.edu}

\author[P. Durcik]{Polona Durcik}
\address{Polona Durcik\\
        Schmid College of Science and Technology\\
        Chapman University\\
        One University Drive\\
        Orange, CA 92866, USA}
\email{durcik@chapman.edu}

\author[V. Kova\v{c}]{Vjekoslav Kova\v{c}}
\address{Vjekoslav Kova\v{c}\\
        Department of Mathematics, Faculty of Science\\
        University of Zagreb\\
        10000 Zagreb, Croatia}
\email{vjekovac@math.hr}

\author[J. Roos]{Joris Roos}
\address{Joris Roos\\
Department of Mathematical Sciences\\
University of Massachusetts Lowell\\
Lowell, MA 01854, USA\\
\& School of Mathematics\\
The University of Edinburgh\\
Edinburgh EH9 3FD, UK}
\email{joris\_roos@uml.edu}

\thanks{The first author was supported by National Science Foundation grant DMS-1901413.
The third author was supported in part by the Croatian Science Foundation under the project IP-2018-01-7491 (DEPOMOS)}


\begin{document}

\begin{abstract}
We prove a.e.\@ convergence of continuous-time quadratic averages with respect to two commuting $\mathbb{R}$-actions, coming from a single jointly measurable measure-preserving $\mathbb{R}^2$-action on a probability space.
The key ingredient of the proof comes from recent work on
multilinear singular integrals; more specifically, from the study of a curved model for the triangular Hilbert transform.
\end{abstract}

\maketitle

\def\R{\mathbb{R}}
\def\Z{\mathbb{Z}}

\numberwithin{equation}{section}


\section{Introduction}
In this note, we apply recent progress in multilinear harmonic analysis \cite{CDR20}
to a problem on convergence almost everywhere in the ergodic theory.

Suppose there to be given an action of the group $\mathbb{R}^2$ on a probability space $(X,\mathcal{F},\mu)$,
\[ \mathbb{R}^2 \times X \to X, \quad (g,x)\mapsto g\cdot x, \]
which is jointly measurable and measure-preserving.
In the language of Varadarajan \cite{V63:basics}, $(X,\mathcal{F})$ is a Borel $\mathbb{R}^2$-space and $\mu$ is an invariant measure.

An alternative way of looking at this setup is to define mutually commuting one-parameter groups of $(\mathcal{F},\mathcal{F})$-measurable measure-$\mu$-preserving transformations $(S^t\colon X\to X)_{t\in\mathbb{R}}$ and $(T^t\colon X\to X)_{t\in\mathbb{R}}$ by
\[ S^t x := (t,0)\cdot x,\quad T^t x := (0,t) \cdot x \]
for every $t\in\mathbb{R}$ and $x\in X$. That way the above $\mathbb{R}^2$-action can be rewritten simply as $((s,t),x) \mapsto S^s T^t x$, but note that we also require joint measurability of this map. On the other hand, $(t,x)\mapsto S^t x$ and $(t,x)\mapsto T^t x$ are two mutually commuting measure-preserving $\mathbb{R}$-actions. We find the latter viewpoint and notation more suggestive, as they emphasize analogies with the corresponding discrete setup, i.e., $\mathbb{Z}^2$-actions, which are determined simply by two commuting transformations $S=S^1$ and $T=T^1$; for example see \eqref{eqn:discreteaverages} and \eqref{eqn:linearaverages} below.

Fix $p,q\in[1,\infty]$ such that $1/p+1/q\leq 1$.
We are interested in the following continuous-time double averages:
\begin{equation}\label{eqn:mainaverages}
A_N(f_1,f_2)(x) := \frac{1}{N} \int_0^N f_1(S^t x) f_2(T^{t^2} x) \,\textup{d}t,
\end{equation}
defined for a positive real number $N$, functions $f_1\in\textup{L}^p(X)$ and $f_2\in\textup{L}^q(X)$, and a point $x\in X$.
If $f_1$ and $f_2$ are given, then for $\mu$-almost every $x$ the integrals in \eqref{eqn:mainaverages} exist and continuously depend on $N\in(0,\infty)$. Indeed, the Tonelli-Fubini theorem, H\"{o}lder's inequality, monotonicity of the $\textup{L}^p(X)$-norms, and the fact that $S^t$, $T^{t^2}$ preserve measure $\mu$, together imply
\[ \int_X \int_{0}^{M} \big| f_1(S^t x) f_2(T^{t^2} x) \big| \,\textup{d}t \,\textup{d}\mu(x) \leq M \|f_1\|_{\textup{L}^p(X)} \|f_2\|_{\textup{L}^q(X)} < \infty \]
for any positive number $M$.
Most of the literature that studies multiple ergodic averages simply takes the functions to be in $\textup{L}^\infty(X)$.

General single-parameter polynomial multiple ergodic averages were introduced by Bergelson and Leibman \cite{BL96:polyaver,BL02:nilaver}, albeit in a discrete setting.
The averages \eqref{eqn:mainaverages} constitute the simplest case of such polynomial (but not purely linear) averages with respect to several commuting group actions. This note establishes their convergence almost everywhere.

\begin{thm}\label{thm:main}
Let $((s,t),x)\mapsto S^s T^t x$ be a jointly measurable measure-preserving action of $\mathbb{R}^2$ on a probability space $(X,\mathcal{F},\mu)$. Let $p,q\in(1,\infty]$ satisfy $1/p+1/q\leq 1$.
Let $f_1\in\textup{L}^p(X)$ and $f_2\in\textup{L}^q(X)$.
Then for $\mu$-almost every $x\in X$ the limit
\[ \lim_{N\to\infty} \frac{1}{N} \int_0^N f_1(S^t x) f_2(T^{t^2} x) \,\textup{d}t \]
exists.
\end{thm}

To the authors' knowledge, this is
the first result on pointwise convergence of some single-parameter multiple ergodic averages with respect to two  general commuting $\mathbb{R}$--actions, without any structural assumptions on the measure-preserving system in question.

Generalizations of continuous-time single-parameter averages \eqref{eqn:mainaverages} to $\mathbb{R}^D$-actions, several polynomials, and several functions were studied by Austin \cite{A12:normpoly}. He showed that these multiple averages always converge in the $\textup{L}^2$-norm when the functions are taken from $\textup{L}^\infty(X)$. The paper \cite{A12:normpoly} also emphasizes simplifications coming from working in the continuous-time setting, as opposed to the discrete one. The most notable simplification comes
from the ability to change variables in integrals with respect to the time-variable. Bergelson, Leibman, and Moreira \cite{BLM12:cont} went a step further by giving general principles for deducing continuous results on convergence of various ergodic averages from their discrete analogues.
A discrete-time analogue of Austin's $\textup{L}^2$-convergence result was later established
(in the greater generality of nilpotent group actions) by Walsh \cite{W12:norm}.

However, pointwise results on single-parameter multiple ergodic averages are much more difficult in either of the two settings. Without further structural assumptions, a.e.\@ convergence is only known for double averages with respect to a single (invertible bi-measurable) measure-preserving transformation $T\colon X\to X$,
\[ \frac{1}{N} \sum_{n=0}^{N-1} f_1(T^{P_1(n)}x) f_2(T^{P_2(n)}x), \]
when either $P_1,P_2$ are both linear polynomials (a result by Bourgain \cite{B90:aedouble}, with its continuous-time analogue formulated explicitly in \cite[Theorem 8.30]{BLM12:cont}) or when $P_1$ is linear and $P_2$ has degree greater than $1$ (a recent result by Krause, Mirek, and Tao \cite{KMT20:aepoly}). The latter case naturally motivates the study of averages
\begin{equation}\label{eqn:discreteaverages}
\frac{1}{N} \sum_{n=0}^{N-1} f_1(S^n x) f_2(T^{n^2}x),
\end{equation}
where $S,T\colon X\to X$ are now two commuting (invertible bi-measurable) measure-preserving transformations. Convergence a.e.\@ of \eqref{eqn:discreteaverages} is still open at the time of writing and Theorem~\ref{thm:main} solves a continuous-time analogue of this problem. As yet another source of motivation we mention that a.e.\@ convergence of purely linear double averages
\begin{equation}\label{eqn:linearaverages}
\frac{1}{N} \sum_{n=0}^{N-1} f_1(S^n x) f_2(T^n x)
\end{equation}
is also a well-known open problem for general commuting $S$ and $T$; see the survey paper by Frantzikinakis \cite{F16:open}. On the other hand, continuous-time analogues of \eqref{eqn:linearaverages} are thought to be equally difficult as \eqref{eqn:linearaverages} themselves: crucial differences disappear in the case of linear powers of transformations.
We remark in passing that a.e.\@ convergence is known for various multi-parameter multiple ergodic averages, such as two types of ``cubic'' averages, see \cite{As10:cubic,CF12:cubic} and \cite{DS16:cubic,DS16:multi}, or ``additionally averaged'' averages, see \cite{HSY19:pt,DS15:pt,DS16:multi}.
Questions on convergence of such averages tend to be easier, but these objects appear naturally in studies of single-parameter averages.

It may be of interest to establish more quantitative variants of Theorem~\ref{thm:main}.
We exploit two nonquantitative reductions:
We use a maximal function inequality combined with convergence on a dense subset
(as opposed to bounding a certain variational norm, as in \cite{B88b:aepoly,B89:aepoly,KMT20:aepoly}),
and we work with lacunary sequences of scales
(as opposed to discussing long and short jumps separately, as in \cite{JSW08:jumps}).

A minor modification of the proof presented here can establish a.e.\@ convergence of variants of the averages \eqref{eqn:mainaverages} in which $t^2$ is replaced with $t^\kappa$ for some fixed positive number $\kappa\neq 1$. Indeed,
for the main technical ingredient of the proof, Theorem \ref{thm:local},
this generalization is sketched in
\cite{CDR20}. The particular choice $\kappa=2$ is also used below in connection
with \eqref{eqn:ergredconv2} and \eqref{eqn:bdelta}, but at those junctures of the proof,
the restriction to $\kappa=2$ is an inessential matter of convenience.

The rest of the paper is dedicated to the proof of Theorem~\ref{thm:main}.
We can assume $p,q\in(1,\infty)$ and $1/p+1/q=1$. Indeed, the $\textup{L}^p$-spaces with respect to a finite measure are nested, which allows raising of either of the two exponents.
Otherwise, the largest range of $(p,q)\in[1,\infty]^2$ in which the a.e.\@ convergence result holds is not clear and even justification of the defining formula \eqref{eqn:mainaverages} is not immediate. A nontrivial $\textup{L}^1$ counterexample for single-function discrete-time quadratic averages was given by Buczolich and Mauldin \cite{BM10:diver}; also see \cite{LaV11:diver} for an extension of their result.

\subsection{Notation}
For two functions $A,B\colon X\to[0,\infty)$ and a set of parameters $P$ we write $A(x) \lesssim_P B(x)$ if the inequality $A(x) \leq C_P B(x)$ holds for each $x\in X$ with a constant $C_P$ depending on the parameters from $P$, but independent of $x$.
Let $\mathbbm{1}_S$ denote the \emph{indicator function} of a set $S\subseteq X$, where the ambient set $X$ is understood from context.
The \emph{floor} of $x\in \R$ will be denoted $\lfloor x\rfloor$; it is the largest integer not exceeding $x$.

If $(X,\mathcal{F},\mu)$ is a measure space and $p\in[1,\infty)$, then the \emph{$\textup{L}^p$-norm} of an $\mathcal{F}$-measurable function $f\colon X\to\mathbb{C}$ is defined as
\[ \|f\|_{\textup{L}^p(X)} := \Big( \int_X |f(x)|^p \,\textup{d}\mu(x) \Big)^{1/p}.\]
We also set
\[ \|f\|_{\textup{L}^\infty(X)} := \mathop{\textup{ess\,sup}}_{x\in X}|f(x)|. \]
On the other hand, the \emph{weak $\textup{L}^p$-norm} is defined as
\[ \|f\|_{\textup{L}^{p,\infty}(X)} := \Big( \sup_{\alpha\in(0,\infty)} \alpha^p \mu\big(\{ x\in X : |f(x)|>\alpha\}\big) \Big)^{1/p}. \]
Occasionally, the variable with respect to which the norm is taken will be denoted in the subscript, so that we can write $\|f(x)\|_{\textup{L}^p_x(X)}$ in place of $\|f\|_{\textup{L}^p(X)}$.
On $\R^d$ the Lebesgue measure will always be understood.

The \emph{Fourier transform} of $f\in\textup{L}^1(\R^d)$ is defined as
\[ \widehat{f}(\xi) := \int_{\R^d} f(x) e^{-2\pi i x\cdot \xi} \,\textup{d}x \]
for each $\xi\in\R^d$, where $(x,y)\mapsto x\cdot y$ is the standard scalar product on $\mathbb{R}^d$.
The map $f\mapsto \widehat{f}$ extends by continuity to the space $\textup{L}^2(\R^d)$, where it becomes a linear isometric isomorphism.

We write $\mathop{\textup{span}}(S)$ for the linear span of a set of vectors $S$ in some linear space. If $V$ and $W$ are mutually orthogonal subspaces of some inner product space, then $V\oplus W$ will denote their \emph{(orthogonal) sum}, i.e., the linear span of their union.
Finally, $\mathop{\textup{img}}(L)$ and $\mathop{\textup{ker}}(L)$ will, respectively, denote the range and the null space of a linear operator $L$.

\section{Ergodic theory reductions}

Theorem~\ref{thm:main} will be deduced from the following proposition dealing with functions on the real line.

\begin{prop}\label{prop:main}
For each $\delta\in(0,1]$ there exists a constant $\gamma\in(0,1)$ such that
\begin{equation}\label{eqn:harmL1bound}
\Big\| \frac{1}{N} \int_0^N \big( F_1(u+t+
\delta,v) - F_1(u+t,v) \big) F_2(u,v+t^2) \,\textup{d}t \Big\|_{\textup{L}^1_{(u,v)}(\mathbb{R}^2)} \lesssim_{\gamma,\delta} N^{-\gamma} \|F_1\|_{\textup{L}^2(\mathbb{R}^2)} \|F_2\|_{\textup{L}^2(\mathbb{R}^2)}
\end{equation}
for every $N\in[1,\infty)$ and all $F_1,F_2\in\textup{L}^2(\mathbb{R}^2)$.
\end{prop}

The proof of Proposition \ref{prop:main} will be postponed until the next section. Moreover, we will see that the quantifiers can be reversed: we will be able to choose $\gamma$ that works for each $\delta$. Here we show how \eqref{eqn:harmL1bound} implies the main result.

\begin{proof}[Proof of Theorem \ref{thm:main}]
Let $p^{-1}+q^{-1}=1$.
We begin by applying
a variant of the so-called \emph{lacunary subsequence trick}; see \cite[Appendix~A]{FLW12:lac}. It reduces Theorem~\ref{thm:main} to proving that
\begin{equation}\label{eqn:subseqconv}
\textup{$\big( A_{\alpha^n}(f_1,f_2)(x) \big)_{n=0}^{\infty}$ converges in $\mathbb{C}$ for a.e.\@ $x\in X$}
\end{equation}
for every fixed $\alpha\in(1,\infty)$.
Indeed, we can assume that $f_1$ and $f_2$ are nonnegative functions, because otherwise we can split them, first into real and imaginary, and then into positive and negative parts.
Denoting by $\lfloor y\rfloor$ the largest integer not exceeding a real number $y$, we can estimate
\[ \alpha^{-1} A_{\alpha^{\lfloor\log_{\alpha}N\rfloor}}(f_1,f_2)(x) \leq A_N(f,g)(x) \leq \alpha A_{\alpha^{\lfloor\log_{\alpha}N\rfloor+1}}(f_1,f_2)(x) \]
and this implies
\begin{align}
\alpha^{-1} \liminf_{\mathbb{N}\ni n\to\infty}A_{\alpha^n}(f_1,f_2)(x) & \leq \liminf_{\mathbb{R}\ni N\to\infty}A_N(f_1,f_2)(x) \nonumber \\
\leq \limsup_{\mathbb{R}\ni N\to\infty}A_N(f_1,f_2)(x) & \leq \alpha \limsup_{\mathbb{N}\ni n\to\infty}A_{\alpha^n}(f_1,f_2)(x). \label{eqn:lacnested}
\end{align}
By \eqref{eqn:subseqconv} applied with $\alpha=2^{2^{-m}}$ we know that at almost every point $x\in X$ the limit
\[ \lim_{n\to\infty}A_{2^{n2^{-m}}}(f_1,f_2)(x) \]
exists for each positive integer $m$. Its value is independent of $m$, since the corresponding sequences are subsequences of each other, so we can denote it by $L(x)\in[0,\infty)$. For any such $x$ the estimate \eqref{eqn:lacnested} gives
\[ 2^{-2^{-m}} L(x) \leq \liminf_{N\to\infty}A_N(f_1,f_2)(x) \leq \limsup_{N\to\infty}A_N(f_1,f_2)(x) \leq 2^{2^{-m}} L(x), \]
so we may
let $m\to\infty$ and conclude that $\lim_{N\to\infty}A_N(f_1,f_2)(x)$ exists and also equals $L(x)$.

We will also use the following easy weak-type inequality:
\begin{equation}\label{eqn:weakmaxbound}
\Big\| \sup_{N\in(0,\infty)} \big| A_N(f_1,f_2) \big| \Big\|_{\textup{L}^{1,\infty}(X)}
\lesssim_{p,q} \|f_1\|_{\textup{L}^{p}(X)} \|f_2\|_{\textup{L}^{q}(X)}
\end{equation}
for every $N\in(0,\infty)$, $f_1\in\textup{L}^p(X)$, and $f_2\in\textup{L}^q(X)$.
It will enable us to restrict attention to dense subspaces of functions $f_1\in\textup{L}^{p}(X)$ and $f_2\in\textup{L}^{q}(X)$ by the aforementioned a.e.\@ convergence paradigm.
In order to prove \eqref{eqn:weakmaxbound} one can first apply H\"{o}lder's inequality, followed by the change of variables $s=t^2$ and a dyadic splitting of the integral in the second term:
\[ \big| A_N(f_1,f_2) \big| \leq
\Big( \frac{1}{N} \int_0^{N}|f_1(S^t x)|^p\,\textup{d}t \Big)^{1/p}
\bigg( \sum_{m=1}^{\infty} 2^{-m/2} \frac{1}{2^{-m+1}N^2} \int_0^{2^{-m+1}N^2} |f_2(T^s x)|^q\,\textup{d}s \bigg)^{1/q}. \]
Then one can take the supremum in $N$ and recall H\"{o}lder's inequality in Lorentz spaces \cite{ON63:Lorentz} to bound the left hand side of \eqref{eqn:weakmaxbound} by
\[ \Big\| \sup_{N\in(0,\infty)} \frac{1}{N} \int_0^{N}|f_1(S^t x)|^p\,\textup{d}t \Big\|_{\textup{L}^{1,\infty}(X)}^{1/p}
\Big\| \sup_{N\in(0,\infty)} \frac{1}{N} \int_0^{N}|f_2(T^t x)|^q\,\textup{d}t \Big\|_{\textup{L}^{1,\infty}(X)}^{1/q}. \]
	It remains to apply the maximal ergodic weak $\textup{L}^1$ inequality to the functions $|f_1|^p$ and $|f_2|^q$. If one only wants to use the well-known discrete-time maximal ergodic theorem, one can borrow a trick from \cite{BLM12:cont}, i.e., restrict the values
of $N$ to the grid $\delta\mathbb{Z}$ for some $\delta>0$ and apply the discrete-time theory to the $\textup{L}^1$ functions
\[ g_1(x) := \frac{1}{\delta} \int_0^{\delta}|f_1(S^t x)|^p\,\textup{d}t, \quad g_2(x) := \frac{1}{\delta} \int_0^{\delta}|f_2(T^t x)|^q\,\textup{d}t. \]
This completes the proof of \eqref{eqn:weakmaxbound}.

A strengthening of \eqref{eqn:weakmaxbound} with the ordinary (strong) $\textup{L}^1$-norm on the left hand side can be deduced by the method of transference from \cite[Theorem 2]{CDR20}, dealing with functions on the real line. We do not need this strengthening here,
since weak-type maximal inequalities are sufficient for the intended purpose of extending a.e.\@ convergence.

A crucial ingredient of the proof of Theorem~\ref{thm:main}
is the following estimate.
	
\begin{lemma}
For each $\delta\in(0,1]$ there exists a constant $\gamma\in(0,1]$ such that
\begin{equation}\label{eqn:ergL1bound}
\Big\| \frac{1}{N} \int_0^N \big( f_1(S^{t+\delta} x) - f_1(S^t x) \big) f_2(T^{t^2} x) \,\textup{d}t \Big\|_{\textup{L}_{x}^{1}(X)}
\lesssim_{\gamma,\delta} N^{-\gamma} \|f_1\|_{\textup{L}^{2}(X)} \|f_2\|_{\textup{L}^{2}(X)}
\end{equation}
for every $N\in[1,\infty)$ and every $f_1,f_2\in\textup{L}^2(X)$.
\end{lemma}

\begin{proof}
We deduce \eqref{eqn:ergL1bound}
from Proposition \ref{prop:main} using the \emph{Calder\'{o}n transference principle} \cite{Cal68:transf}.
By homogeneity it is sufficient to prove the inequality
\eqref{eqn:ergL1bound}
for functions $f_1$ and $f_2$ normalized to satisfy
\[ \|f_1\|_{\textup{L}^{2}(X)} = \|f_2\|_{\textup{L}^{2}(X)} = 1. \]
For each $x\in X$ and $N\geq 1$
define functions $F_1^{x,N},F_2^{x,N}\colon\mathbb{R}^2\to\mathbb{C}$
by
\[ F_j^{x,N}(u,v) := f_j(S^u T^v x) \mathbbm{1}_{[0,3N]}(u) \mathbbm{1}_{[0,2N^2]}(v) \]
for $(u,v)\in\mathbb{R}^2$ and $j=1,2$. Since the measure $\mu$ is invariant under the $\mathbb{R}^2$-action in question, we can rewrite the left hand side of \eqref{eqn:ergL1bound} as
\begin{align*}
& \frac{1}{N^3} \int_{0}^{N} \int_{0}^{N^2} \int_X \Big| \frac{1}{N} \int_0^N \big( f_1(S^{t+\delta} S^u T^v x) - f_1(S^t S^u T^v x) \big) f_2(T^{t^2} S^u T^v x) \,\textup{d}t \Big| \,\textup{d}\mu(x) \,\textup{d}u \,\textup{d}v \\
& \leq \frac{1}{N^3} \int_X \Big\| \frac{1}{N} \int_0^N \big( F_1^{x,N}(u+t+
\delta,v) - F_1^{x,N}(u+t,v) \big) F_2^{x,N}(u,v+t^2) \,\textup{d}t \Big\|_{\textup{L}^1_{(u,v)}(\mathbb{R}^2)} \,\textup{d}\mu(x).
\end{align*}
An application of \eqref{eqn:harmL1bound} with functions $F_1^{x,N}, F_2^{x,N}$ for each fixed $x\in X$ bounds the last display by a constant multiple of
\begin{align*}
& \frac{1}{N^3} \int_X N^{-\gamma} \frac{1}{2} \big( \big\|F_1^{x,N}\big\|^2_{\textup{L}^2(\mathbb{R}^2)} + \big\|F_2^{x,N}\big\|^2_{\textup{L}^2(\mathbb{R}^2)} \big) \,\textup{d}\mu(x) \\
& = N^{-\gamma} \frac{1}{N^3} \int_{0}^{3N} \int_{0}^{2N^2} \int_X \frac{1}{2} \big( |f_1(S^u T^v x)|^2 + |f_2(S^u T^v x)|^2 \big) \,\textup{d}\mu(x) \,\textup{d}u \,\textup{d}v \\
& = 6 N^{-\gamma} \frac{1}{2} \big(\|f_1\|_{\textup{L}^{2}(X)}^2 + \|f_2\|_{\textup{L}^{2}(X)}^2 \big)
= 6 N^{-\gamma},
\end{align*}
where we have again used the invariance of $\mu$. This completes the proof of \eqref{eqn:ergL1bound}.
\end{proof}

For each $t\in\mathbb{R}$ let $U^t$ denote the unitary operator on $\textup{L}^2(X)$ given by the formula $U^t f := f\circ S^t$. Our final auxiliary claim is that
\begin{equation}\label{eqn:densesub}
\mathop{\textup{span}}\Big( \bigcup_{\delta\in(0,1]}  \mathop{\textup{img}}(U^\delta-I) \Big) \oplus \Big( \bigcap_{\delta\in(0,1]} \mathop{\textup{ker}}(U^\delta-I) \Big)
\end{equation}
is a dense subspace of $\textup{L}^2(X)$. Indeed, this easily follows from $\mathop{\textup{img}}(U^\delta-I) ^\bot = \mathop{\textup{ker}}(U^\delta-I)$ for each $\delta$, which, in turn, is a consequence of the fact that $U^\delta-I$ is a normal operator.

\medskip
We are now ready to complete the proof of Theorem~\ref{thm:main}.
By the initial reduction and the maximal inequality \eqref{eqn:weakmaxbound} we
need only establish
\eqref{eqn:subseqconv} for each fixed $\alpha\in(1,\infty)$ and for functions $f_1,f_2\in\textup{L}^2(X)$. The reason is, of course, that $\textup{L}^p(X)\cap\textup{L}^2(X)$ is dense in $\textup{L}^p(X)$, while $\textup{L}^q(X)\cap\textup{L}^2(X)$ is dense in $\textup{L}^q(X)$. By yet another application of \eqref{eqn:weakmaxbound}, this time with $p=q=2$, we see that it suffices to take $f_1$ from the dense subspace \eqref{eqn:densesub} of $\textup{L}^2(X)$. In other words, we can assume that $f_1$ is of the form
\[ \sum_{k=1}^{m} (g_k\circ S^{\delta_k}-g_k) + h, \]
where $m\in\mathbb{N}$, $\delta_1,\ldots,\delta_m\in(0,1]$,
$g_1,\ldots,g_m,h\in\textup{L}^2(X)$, and $h$ is such that $h\circ S^{t}=h$ for each $t\in(0,1]$, and thus also for each $t\in[0,\infty)$. That way the theorem is reduced to showing that
for any $f_1,f_2\in \textup{L}^2(X)$ and any parameters $\alpha>1$ and $\delta\in(0,1]$,
the two sequential limits
\begin{equation}\label{eqn:ergredconv1}
\lim_{n\to\infty} \frac{1}{\alpha^n} \int_0^{\alpha^n} \big( f_1(S^{t+\delta} x) - f_1(S^t x) \big) f_2(T^{t^2} x) \,\textup{d}t
\end{equation}
and
\begin{equation}\label{eqn:ergredconv2}
\lim_{n\to\infty} \frac{1}{\alpha^n} \int_0^{\alpha^n} f_2(T^{t^2} x) \,\textup{d}t
\end{equation}
exist (in $\mathbb{C}$) for a.e.\@ $x\in X$.

Estimate \eqref{eqn:ergL1bound} applied with $N=\alpha^n$ and summation in $n$ give
\[ \int_X \sum_{n=0}^{\infty} \Big| \frac{1}{\alpha^n} \int_0^{\alpha^n} \big( f_1(S^{t+\delta} x) - f_1(S^t x) \big) f_2(T^{t^2} x) \,\textup{d}t \Big| \,\textup{d}\mu(x)
\lesssim_{\gamma,\delta} \sum_{n=0}^{\infty} \alpha^{-\gamma n} \|f_1\|_{\textup{L}^{2}(X)} \|f_2\|_{\textup{L}^{2}(X)} < \infty. \]
Thus, for a.e.\@ $x\in X$ the sequence in \eqref{eqn:ergredconv1} converges to $0$, as a general term of a convergent series.

The limit in \eqref{eqn:ergredconv2} exists for a.e.\@ $x\in X$ by \cite[Theorem 8.31]{BLM12:cont}, which claims the same for general polynomial averages of a single $\textup{L}^2$ function and constitutes a continuous-time analogue of Bourgain's result from \cite{B88b:aepoly}.
\end{proof}

\section{Harmonic analysis reductions}

\begin{proof}[Proof of Proposition \ref{prop:main}]
Let $\zeta$ be a $C^\infty$ function compactly supported in $\R^2\times (\R\setminus\{0\})$.

\begin{thm}[\cite{CDR20}] \label{thm:local}
There exist $C,\sigma>0$ with the following property.
Let $F_1,F_2\in \textup{L}^2(\mathbb R^2)$ and let $\lambda\ge 1$.
Suppose that, for at least one of the indices $j=1,2$,
$\widehat{F}_j(\xi_1,\xi_2)$ vanishes whenever $|\xi_j| < \lambda$.
Then
\[ \Big\| \int_{\R} F_1(x+t,y)\, F_2(x,y+t^2)\, \zeta(x,y,t) \,\textup{d}t \Big\|_{\textup{L}^1_{(x,y)}(\R^2)}  \le C \lambda^{-\sigma} \|F_1\|_{\textup{L}^2(\R^2)}
\|F_2\|_{\textup{L}^2(\R^2)}. \]
\end{thm}

For an auxiliary function $\zeta$ as before, any $\delta\in (0,1]$, and any $F_1,F_2\in \textup{L}^2(\R^2)$ define
\begin{equation} \label{eqn:bdelta}
B_\delta(F_1,F_2)(x,y) := \int_{\R} \big( F_1(x+t+\delta,y)-F_1(x+t,y) \big)\,
F_2(x,y+t^2) \, \zeta(x,y,t) \,\textup{d}t .
\end{equation}
We claim that, to prove
Proposition \ref{prop:main}, it suffices to prove that there exists $\gamma\in (0,1)$ such that
\begin{equation}\label{eqn:mainclaim7}
\|B_\delta(F_1,F_2)\|_{\textup{L}^1(\R^2)} \le C_{\gamma,\zeta} \delta^{\gamma} \|F_1\|_{\textup{L}^2(\R^2)} \|F_2\|_{\textup{L}^2(\R^2)}
\end{equation}
for every $\delta\in (0,1]$, for all $F_1,F_2\in \textup{L}^2(\R^2)$, where $C_{\gamma,\zeta}$ is a constant depending on $\gamma$ and $\zeta$.

This is a standard reduction, but some care needs to be taken
due to the minus sign appearing in $B_\delta(F_1,F_2)$.
By using the equality
\[\frac{1}{N} \mathbbm{1}_{(0,N]} = \sum_{k=1}^\infty 2^{-k} \frac{1}{2^{-k}N} \mathbbm{1}_{(2^{-k}N, 2^{-k+1} N]}\]
and rescaling
\[ F_j(x,y) \mapsto (2^{-k}N)^{3/2} F_j\big(2^{-k}N x,(2^{-k}N)^2 y\big), \quad \delta\mapsto(2^{-k}N)^{-1}\delta, \]
the inequality \eqref{eqn:harmL1bound} follows if we can show existence of $\gamma\in (0,1)$ such that
\begin{equation}\label{eqn:subclaim7}
\Big\| \int_1^2 \big( F_1(x+t+
\delta,y) - F_1(x+t,y) \big) F_2(x,y+t^2) \,\textup{d}t \Big\|_{\textup{L}^1_{(x,y)}(\mathbb{R}^2)} \lesssim_{\gamma} \delta^{\gamma} \|F_1\|_{\textup{L}^2(\mathbb{R}^2)} \|F_2\|_{\textup{L}^2(\mathbb{R}^2)}
\end{equation}
for all $\delta>0$. Since \eqref{eqn:subclaim7} is trivial for $\delta>1$ by the Cauchy--Schwarz inequality, we can again assume that $\delta\in(0,1]$.
Next, let $\eta$ be a smooth non-negative function supported in $[-1,1]^2$ and such that $\sum_{m\in\Z^2}\eta_m=1$, where $\eta_m(x,y) := \eta((x,y)-m)$ for all $(x,y)\in\R^2$.
The left hand side of \eqref{eqn:subclaim7} is majorized by
\[ \sum_{m\in\Z^2} \Big\| \int_1^2 \big( (\widetilde{\eta}_m F_1) (x+t+
\delta,y) - (\widetilde{\eta}_m F_1)(x+t,y) \big) (\widetilde{\eta}_m F_2)(x,y+t^2) \eta_m(x,y)\,\textup{d}t \Big\|_{\textup{L}^1_{(x,y)}(\mathbb{R}^2)},
\]
where $\widetilde{\eta}$ is a smooth non-negative function compactly supported in $[-20,20]^2$, equal to $1$ on $[-10,10]^2$ and $\widetilde{\eta}_m(x,y):=\widetilde{\eta}((x,y)-m)$.
To apply \eqref{eqn:mainclaim7} we also need to pass to a smooth cutoff function in the $t$-variable. To this end choose a smooth non-negative function $\varphi$ compactly supported in $[1,2]$ so that $\|\varphi - \mathbbm{1}_{[1,2]} \|_{\textup{L}^1(\R)} \le \delta$.
Applying \eqref{eqn:mainclaim7} with $\zeta(x,y,t)=\eta(x,y)\varphi(t)$ and majorizing the error term by the Minkowski and Cauchy--Schwarz inequalities shows that the previous display is majorized by
\[ (C_{\gamma,\zeta} \delta^{\gamma} + \delta) \sum_{m\in\Z^2} \|\widetilde{\eta}_m F_1\|_{\textup{L}^2(\R^2)} \|\widetilde{\eta}_m F_2\|_{\textup{L}^2(\R^2)}. \]
By the Cauchy--Schwarz inequality for the sum in $m$, the previous display is at most a constant multiple of $\delta^\gamma \|F_1\|_{\textup{L}^2(\R^2)} \|F_2\|_{\textup{L}^2(\R^2)}$, which proves the claim, i.e., establishes Proposition \ref{prop:main},
modulo the proof of \eqref{eqn:mainclaim7}.

\begin{proof}[Proof of \eqref{eqn:mainclaim7}]
Let $R\ge 1$ be determined later. Decompose
\[ F_1 = F_{1,R} + G_{1,R}, \]
where $F_{1,R}$ is defined via its Fourier transform as
\[ \widehat{F_{1,R}}(\xi_1,\xi_2) = \widehat{F_1}(\xi_1,\xi_2) \mathbbm{1}_{[-R,R]}(\xi_1) \]
for each $(\xi_1,\xi_2)\in\mathbb{R}^2$.
With $B_\delta$ defined by \eqref{eqn:bdelta}, split
\begin{equation}\label{eqn:thesplitting}
B_{\delta}(F_1,F_2) = B_\delta(F_{1,R}, F_2) +   B_\delta(G_{1,R}, F_2).
\end{equation}
Using Theorem \ref{thm:local} we estimate
\begin{equation}\label{eqn:auxsplitineq1}
\| B_\delta(G_{1,R}, F_2)\|_{\textup{L}^1(\R^2)}
\lesssim_{\zeta} R^{-\sigma} \|G_{1,R}\|_{\textup{L}^2(\R^2)} \|F_2\|_{\textup{L}^2(\R^2)}
\leq R^{-\sigma} \|F_1\|_{\textup{L}^2(\R^2)} \|F_2\|_{\textup{L}^2(\R^2)}
\end{equation}
with $\sigma > 0$.
It remains to control $B_{\delta}(F_{1,R}, F_2)$. 
Applying the Cauchy--Schwarz inequality in $(x,y)$ for each fixed $t$, we obtain
\[ \|B_{\delta}(F_{1,R},F_2)\|_{\textup{L}^1(\R^2)} \lesssim_{\zeta} \|F_{1,R}(x+\delta,y)-F_{1,R}(x,y)\|_{\textup{L}_{(x,y)}^2(\R^2)} \|F_2\|_{\textup{L}^2(\R^2)}. \]
The Plancherel identity gives
\[ \|F_{1,R}(x+\delta,y)-F_{1,R}(x,y)\|_{\textup{L}_{(x,y)}^2(\R^2)}^2 =\int_{[-R,R]\times\R} \big|\widehat{F_1}(\xi_1,\xi_2)\big|^2 \big|e^{2\pi i \delta \xi_1}-1\big|^2 \,\textup{d}\xi_1 \,\textup{d}\xi_2, \]
while $|\xi_1|\leq R$ implies $|e^{2\pi i \delta \xi_1}-1|\lesssim \delta R$.
Therefore,
\begin{equation}\label{eqn:auxsplitineq2}
\|B_{\delta}(F_{1,R},F_2)\|_{\textup{L}^1(\R^2)} \lesssim_{\zeta} \delta R \,\|F_1\|_{\textup{L}^2(\R^2)} \|F_2\|_{\textup{L}^2(\R^2)}.
\end{equation}
From \eqref{eqn:auxsplitineq1}, \eqref{eqn:auxsplitineq2}, and the splitting \eqref{eqn:thesplitting} we finally conclude
\[ \|B_{\delta}(F_1,F_2)\|_{\textup{L}^1(\R^2)} \lesssim_{\zeta} ( \delta R + R^{-\sigma} ) \|F_1\|_{\textup{L}^2(\R^2)} \|F_2\|_{\textup{L}^2(\R^2)}, \]
so we are done by choosing $R=\delta^{-1/2}$ and $\gamma=\min\{1/2,\sigma/2\}$.
\end{proof}

This completes the proof of Proposition~\ref{prop:main}.
\end{proof}


\section*{Acknowledgment}
The authors are grateful to Terence Tao for
raising the question answered here, and for
pointing out its connection with \cite{CDR20}.


\end{document}